\documentclass[11pt]{amsart}
\usepackage{mathrsfs,amsthm,amsmath,amssymb,bbm,pdfsync,graphicx,fullpage,esint}
\usepackage{prettyref,hyperref,varioref}

\usepackage{wasysym}

\theoremstyle{plain}
\newtheorem{thm}{\protect\theoremname}[section]
\newtheorem*{thm*}{\protect\theoremname}

\newtheorem{lem}[thm]{\protect\lemmaname}
\newtheorem*{lem*}{\protect\lemmaname}

\newtheorem{prop}[thm]{\protect\propositionname}
\newtheorem*{prop*}{\protect\propositionname}

\newtheorem{cor}[thm]{\protect\corollaryname}
\newtheorem*{cor*}{\protect\corollaryname}

\newtheorem*{fact*}{\protect\factname}

\theoremstyle{definition}

\newtheorem*{defn*}{\protect\definitionname}

\newtheorem{quest}[thm]{\protect\questionname}
\theoremstyle{remark}
\newtheorem{rem}[thm]{\protect\remarkname}

\numberwithin{equation}{subsection}

\newcommand{\C}{\mathbb{C}}

\newcommand{\indicator}[1]{\mathbbm{1}_{#1}}

\newcommand{\at}[1]{{#1}\rvert}

\providecommand{\corollaryname}{Corollary}
\providecommand{\definitionname}{Definition}
\providecommand{\factname}{Fact}
\providecommand{\lemmaname}{Lemma}
\providecommand{\propositionname}{Proposition}
\providecommand{\theoremname}{Theorem}
\providecommand{\remarkname}{Remark}
\providecommand{\conjecturename}{Conjecture}
\providecommand{\examplename}{Example}
\providecommand{\claimname}{Claim}
\providecommand{\problemname}{Problem}
\providecommand{\questionname}{Question}

\newrefformat{prop}{Proposition \ref{#1}}
\newrefformat{subsub}{Section \ref{#1}}
\newrefformat{fact}{Fact \ref{#1}}
\newrefformat{cor}{Corollary \ref{#1}}
\newrefformat{exam}{Example \ref{#1}}
\newrefformat{rem}{Remark \ref{#1}}
\newrefformat{defn}{Definition \ref{#1}}
\newrefformat{conj}{Conjecture \ref{#1}}
\newrefformat{quest}{Question \ref{#1}}
\newrefformat{sub}{Section \ref{#1}}

\begin{document}

\title{Bounding the Complexity of Replica Symmetry Breaking for Spherical Spin Glasses}

\author{Aukosh Jagannath}
\address[Aukosh Jagannath]{Courant Institute of Mathematical Sciences, 251 Mercer St.\ NY, NY, USA, 10012}
\email{aukosh@cims.nyu.edu}

\author{Ian Tobasco}
\address[Ian Tobasco]{Courant Institute of Mathematical Sciences, 251 Mercer St.\ NY, NY, USA, 10012}
\email{tobasco@cims.nyu.edu}

\begin{abstract}
In this paper, we study the Crisanti-Sommers variational problem, which
is a variational formula for the free energy of spherical mixed $p$-spin
glasses. We begin by computing the dual of this problem using a min-max
argument. We find that the dual is a 1-D problem of obstacle type,
where the obstacle is related to the covariance structure of the underlying
process. This approach yields an alternative way to understand Replica
Symmetry Breaking at the level of the variational problem through
topological properties of the coincidence set of the optimal dual
variable. Using this duality, we give an
algorithm to reduce this \emph{a priori} infinite dimensional variational problem
to a finite dimensional one, thereby confining all possible forms of
Replica Symmetry Breaking in these models to a finite parameter family.
These results complement the authors' related results for the low
temperature $\Gamma$-limit of this variational problem. We briefly
discuss the analysis of the Replica Symmetric phase using this approach.
\end{abstract}

\date{July 7, 2016}

\maketitle

\section{Introduction}

The analysis of variational formulas for free energies plays a central
role in statistical physics. The phenomenological properties of their
optimizers are of particular interest as they play the role of an
order parameter, encoding physical properties of the system. In the
study of mean field spin glasses, these formulas traditionally come
in the form of a strictly convex minimization problem over the space
of probability measures. Put briefly, the minimizers of these problems
are related to the laws of relative positions of configurations drawn
from a corresponding Gibbs measure. Changes in the topology of their
support are related to the presence and nature of a spin glass phase.
For more on this see \cite{JagTobPD15,MPV87,PanchSKBook}.

Though there has been major progress in the mathematical study of
variational problems arising from mean field spin glasses in recent
years \cite{AuffChen13,AuffChen14,JagTobLT16,JagTobPD15}, many interesting
questions remain open. In this paper, we present an alternative perspective
on a class of these problems, specifically those coming from the spherical
mixed $p$-spin glass models. In this class, the variational problem
involves the \emph{Crisanti-Sommers functional} which we define presently.

Let $\xi(t)=\beta^{2}\sum_{p\geq2}\beta_{p}^{2}t^{p}$ satisfy $\xi(1+\epsilon)<\infty$
and let $h$ be a non-negative real number. Call $\xi_{0}=\beta^{-2}\xi$.
The parameters $\xi_{0}$, $\beta$, and $h$ are called the \emph{model},
\emph{inverse temperature, }and \emph{external field} respectively.
For $\mu\in\Pr([0,1])$, let $\phi_{\mu}(t)=\int_{t}^{1}\mu([0,s])ds$.
The Crisanti-Sommers\emph{ }functional is then defined as

\[
P_{\xi,h}(\mu)=\frac{1}{2}\left(\int_{0}^{1}\xi''\phi_{\mu}\,ds+\int_{0}^{1}\left(\frac{1}{\phi_{\mu}}-\frac{1}{1-s}\right)\,ds+h^{2}\phi_{\mu}(0)\right).
\]
As $\phi_{\mu}(s)\leq1-s$, the second term is non-negative so that
this functional is well defined. Observe that the functional is strictly
convex by the strict convexity of the function $x\mapsto\frac{1}{x}$,
and is lower semi-continuous in the weak-$*$ topology. (For experts:
this is a lower semi-continuous extension of the functional originally
defined by Crisanti and Sommers \cite{CriSom92}. See \cite[Section 6.1]{JagTobLT16}
for more details.) 

The Crisanti-Sommers variational problem is given by

\begin{equation}
F(\xi,h)=\min_{\mu\in\Pr([0,1])}P_{\xi,h}(\mu).\label{eq:Crisanti-Sommers-var-prob-def}
\end{equation}
This gives a variational formula for the free energy in spherical
mixed $p$-spin glasses. In particular, for $\sigma\in S^{N-1}(\sqrt{N})$,
if we define the Hamiltonian

\begin{equation}
H_{N}(\sigma)=\beta\sum_{p=2}^{\infty}\frac{\beta_{p}}{N^{\frac{p-1}{2}}}\sum_{i_{1},\ldots,i_{p}=1}^{N}g_{i_{1},\ldots,i_{p}}\sigma_{i_{1}}\cdots\sigma_{i_{p}}+h\sum_{i=1}^{N}\sigma_{i},\label{eq:spherical-model-ham}
\end{equation}
where $g_{i_{1}\ldots i_{p}}$ are i.i.d.\ standard normal random
variables, then 
\[
F(\xi,h)=\lim_{N\to\infty}\frac{1}{N}\log\int_{S^{N-1}(\sqrt{N})}e^{H_{N}(\sigma)}dvol_{N-1}\quad a.s.
\]
where $dvol_{N-1}$ is the normalized volume form on $S^{N-1}(\sqrt{N}$).
This was first established by Crisanti and Sommers \cite{CriSom92}
non-rigorously using the Replica method. It was proved by Talagrand
\cite{TalSphPF06} in the case that $\xi$ is even and Chen \cite{ChenSph13}
for general $\xi$. 

Of particular interest in the analysis of the variational problem
\prettyref{eq:Crisanti-Sommers-var-prob-def} is the development of
its phase diagram. In particular, one is interested in characterizing
the regions in $(\xi,h)$-space where the minimizer, $\mu$, has
\begin{itemize}
\item 1 atom, called Replica Symmetry (RS)
\item $k+1$ atoms, called $k$-step Replica Symmetry Breaking (kRSB)
\item an absolutely continuous part, called full Replica Symmetry Breaking
(fRSB).
\end{itemize}
One is also interested in the nature of the fRSB phase, e.g., whether
there are many absolutely continuous components in the support of
$\mu$, and whether there are atoms in between. 

The notion of Replica Symmetry Breaking (RSB) was introduced by Parisi
in his groundbreaking papers \cite{Par79}. The physical interpretation
and mathematical foundations of RSB has been a subject of great research
for several decades in both communities. 
In many models, and in particular \eqref{eq:spherical-model-ham}, there is a rigorous mathematical interpretation
of RSB that depends on the form of the optimizer of \prettyref{eq:Crisanti-Sommers-var-prob-def}.  
For a physically minded discussion of this see \cite{MPV87}, and
for a mathematically minded discussion of this see \cite{PanchSKBook}. 

With this picture in mind, it then becomes of paramount interest to
determine which models have what forms of RSB. Until now, this question
has been studied by either: proposing an ansatz --- perhaps through some deep
physical insight --- and analytically testing its validity; numerically
optimizing the problem; or performing some combination of the two.
Ideally, however, one would be able to simply look at an Hamiltonian
and provide a simple \emph{a priori} bound on the nature and complexity
of RSB that can arise (perhaps after a back of the envelope calculation).

In this paper, we provide such a technique. We give a simple test
which, given the form of the Hamiltonian, reduces the nature and form
of the RSB to a simple finite parameter family. 
This family is given explicitly in terms of the model in \prettyref{cor:RSBreduction}.

We now turn to a more precise discussion of the methods introduced
in this paper and their relation to the previous literature. Previously,
the mathematical analysis of the Crisanti-Sommers problem focused on the study of the variational inequality 
\begin{equation}
\left(\eta(s)-\xi(s),\nu-\mu\right)\geq0\quad\forall\,\nu\in\Pr([0,1])\label{eq:var-ineq}
\end{equation}
which arises from its first order optimality conditions. Here, $\eta$
depends on the measure $\mu$. The focus of this analysis is generally
to test and rule out ansatzes. As $P$ is strictly convex, one can
prove that an ansatz is optimal by checking that \eqref{eq:var-ineq}
is satisfied. The non-locality of the map $\mu\mapsto\eta_{\mu}$
is a major technical difficulty in the study of RSB: the more complex
the ansatz, the more unwieldy this approach becomes.

In this paper, we take a different approach. It is a classical observation
that variational inequalities such as \eqref{eq:var-ineq} can be
encoded in an \emph{obstacle problem}, that is, a problem of the form
\[
\max_{\eta\geq\psi}D(\eta).
\]
Here, $\psi$ is called the obstacle and $D$ is a local functional
(for this in our setting, see \prettyref{sub:Duality-intro}). This
obstacle problem will be dual to \eqref{eq:Crisanti-Sommers-var-prob-def},
which we call the \emph{primal problem}, and the dual variable will
be $\eta$. The relation between the optimizers of the primal and
dual problems is given by the case of equality in a certain Legendre
transform. 

Two new approaches emerge by allowing $\eta$ to vary. Firstly,
one can analyze the dual problem in and of itself. Alternatively,
one can simultaneously analyze pairs $(\mu,\eta)$ searching for a
pair that satisfies a saddle-type condition. Furthermore, this new perspective
permits methods and techniques which are natural for the dual
problem and have no apparent analogues for the primal one. (See, e.g.,
the proof of \prettyref{thm:Rule-of-Signs} and \prettyref{rem:We-are-awesome}.)

As the dual is an obstacle problem, its first order optimality conditions
will generate a measure that is supported on the \emph{coincidence
set,} $\{s:\eta=\psi\}$. Through the relation between the primal
and dual variables, this measure will solve the primal problem \eqref{eq:Crisanti-Sommers-var-prob-def}. 

This duality yields an alternative perspective on the study of Replica
Symmetry Breaking. The nature and form of the RSB
in a given system can be analyzed by studying the fine topological
properties of the coincidence set of the dual, thereby connecting
a question of pure spin glass theory to a question naturally arising
in the study of free boundary problems in classical physics. Furthermore,
this gives a natural interpretation for $\xi$ --- which describes the
covariance structure of $H_{N}$ --- at the level of the dual: it is
the obstacle, $\psi=\xi$. 

This duality approach was introduced by the authors in \cite{JagTobLT16}
in the low temperature (large $\beta$) limit. There, after developing
the\textbf{ }low temperature \textbf{$\Gamma-$}limit of the Crisanti-Sommers
problem, we obtained its dual problem. By analyzing the two problems
simultaneously, we were able to unify independent conjectures in the
mathematics and physics literatures regarding 1RSB at zero temperature.
Furthermore, we obtained a simple functional form for the optimizer
depending on only finitely many parameters, similar to the one given
below in \prettyref{cor:RSBreduction}. 

In this paper, we find that the duality approach can in fact
be applied at any temperature and external magnetic field, and for
all models. We find that the dual of the Crisanti-Sommers variational
problem is also a maximization problem of obstacle type, though of
a more complicated form than that in \cite{JagTobLT16}. As our
main application of this method, we present the aforementioned reduction
of RSB to a simple finite parameter family; we also discuss some simple
observations regarding RS in these systems.

\subsection{Duality\label{sub:Duality-intro}}

The first step in our approach to the Crisanti-Sommers variational
problem \eqref{eq:Crisanti-Sommers-var-prob-def} is to obtain its
dual. To state the dual problem we need the following notation. Define
the vector space 
\[
X=\{\eta:\eta\in C^{1}([0,1)),\eta'\in Lip_{\text{loc}}([0,1))\}.
\]
Here we allow $\eta$ to blow up at the point $1$. Observe that for
$\eta\in X$, the second distributional derivative, $\eta''$, can
be identified with a measurable function on $[0,1]$ which is a.e.\
unique, and which is essentially bounded on subintervals of the form
$[0,1-\delta]$ where $\delta>0$. 

Let 
\[
\Lambda_{\xi,h}=\left\{ \eta\in X:\eta\geq\xi,\eta'(0)=-h^{2},\eta''\geq0\text{ Leb.-a.e.}\right\} 
\]
and define the functional $D_{\xi,h}:\Lambda_{\xi,h}\to[-\infty,\infty)$
by 
\[
D_{\xi,h}(\eta)=\frac{1}{2}\left(\int_{0}^{1}2\sqrt{\eta''(s)}-\eta''(s)\cdot(1-s)-\frac{1}{1-s}\, ds-\eta(0)+h^{2}+\xi(1)\right).
\]
The integrand is measurable and, by the arithmetic-geometric inequality,
it is non-positive so that its integral is well-defined. 

Finally, we define the\emph{ coincidence set},\emph{ 
\[
\left\{ s\in[0,1):\eta(s)=\xi(s)\right\} .
\]
}
\begin{thm}
(Duality) \label{thm:duality}We have that 
\[
\min_{\mu\in\Pr([0,1])}P_{\xi,h}(\mu)=\max_{\eta\in\Lambda_{\xi,h}}D_{\xi,h}(\eta).
\]
Furthermore, there is a unique optimal pair $(\mu,\eta)\in\Pr([0,1])\times\Lambda_{\xi,h}$,
which is characterized by the relations 
\begin{equation}
\begin{cases}
\mu\left(\{\eta=\xi\}\right)=1\\
\eta''\phi_{\mu}^{2}=1 & Leb-a.e.
\end{cases}\label{eq:opt-conds}
\end{equation}

\end{thm}
One can consider the optimality conditions \prettyref{eq:opt-conds}
in various ways. For example, one can take $\mu$ as the variable
to be optimized while constraining $\eta$ to be given by 
\[
\eta_{\mu}=c-h^{2}t+\int_{0}^{t}\int_{0}^{s}\frac{1}{\phi_{\mu}^{2}}\, d\tau ds,
\]
where $c$ is the number such that $\inf\{\eta_{\mu}-\xi\}=0$.
Then, the optimal $\mu$ is the probability measure which gives full
mass to the coincidence set. This is equivalent to the primal approach
to \prettyref{eq:Crisanti-Sommers-var-prob-def}, i.e., the study
of the variational inequality \prettyref{eq:var-ineq}, and is also equivalent to Talagrand\textquoteright s
characterization of the optimal $\mu$ from \cite[Proposition 2.1]{TalSphPF06}.

A fundamentally different approach to \prettyref{eq:Crisanti-Sommers-var-prob-def}
is to analyze its dual:

\begin{equation}
F(\xi,h)=\max_{\eta\in\Lambda_{\xi,h}}D_{\xi,h}(\eta).\label{eq:dual-prob}
\end{equation}
The next result is regarding optimality and regularity for the dual
variable, $\eta$.
\begin{prop}
\label{prop:dual-regularity} Let $\eta$ be optimal for \prettyref{eq:dual-prob}.
Then,
\begin{itemize}
\item Optimality: 
\begin{equation}
\left(\frac{1}{\sqrt{\eta''}}\right)''=-\mu\label{eq:optimaletaeqn}
\end{equation}
in the sense of distributions where $\mu$ solves \eqref{eq:Crisanti-Sommers-var-prob-def}.
Moreover, the support of $\mu$ is contained in the coincidence set.
\item Regularity:

\begin{enumerate}
\item $\eta\in C^{2}([0,1))$, $\eta'''\in BV((0,1-\delta))\cap L^{\infty}((0,1-\delta))$
for all $\delta\in(0,1)$ 
\item $\frac{1}{(1-s)^{2}}\leq\eta''\leq\frac{C(\xi,h)}{(1-s)^{2}}$ for
some positive constant $C$
\item $\eta''=\frac{1}{(1-s)^{2}}$ sufficiently close to $1$
\end{enumerate}
\item Consistency: if $\eta(s)=\xi(s)$, then 
\begin{equation}
\eta'(s)=\xi'(s)\quad\text{and}\quad\eta''(s)\geq\xi''(s).\label{eq:consistency}
\end{equation}

\end{itemize}
\end{prop}
\begin{proof}
We begin with the first order optimality conditions. Observe that
by the previous theorem, in particular by \prettyref{eq:opt-conds},
we have that
\[
\frac{1}{\sqrt{\eta''}}=\phi_{\mu}
\]
a.e.\ where $\mu$ solves \eqref{eq:Crisanti-Sommers-var-prob-def}.
Differentiating twice we see that \prettyref{eq:optimaletaeqn} holds.
The desired condition on the support of $\mu$ also follows from \prettyref{eq:opt-conds}. 

Now we prove regularity. We begin with claim (1). Observe that 
\[
\eta''=\frac{1}{\phi_{\mu}^{2}}
\]
a.e. Since $\phi_{\mu}$ is continuous and strictly positive away
from $1$, it follows that $\eta\in C^{2}([0,1))$. Differentiating
again, we find that
\[
\eta'''=-\frac{\mu([0,s])}{\phi_{\mu}^{2}}
\]
in the sense of distributions. Note that the righthand side belongs
to $L^{\infty}((0,1-\delta))\cap BV((0,1-\delta))$ for all $\delta\in(0,1)$.
Therefore, so does $\eta'''$. To prove the remaining claims, we will
use that the optimal $\mu$ is supported away from $1$ (cf.\ \prettyref{lem:mu-in-Q}).
This implies that $\phi_{\mu}=1-s$ on a neighborhood of $1$, so
that claim (3) holds. Since $\eta''$ is continuous on $[0,1)$, it
is bounded on each sub-interval of the form $[0,1-\delta]$ where
$\delta\in(0,1)$. Combining this with claim (3) we deduce the upper
bound part of claim (2); the lower bound part follows since in general
$\phi_{\mu}\leq1-s$. 

Finally, we prove the consistency conditions. By its definition, the
coincidence set consists of minimizers of $g=\eta-\xi$ on $[0,1)$.
Since $g\in C^{2}([0,1))$, it follows that $g'=0$ and $g''\geq0$
at such points. This proves \prettyref{eq:consistency}.
\end{proof}
This result shows how to recover the optimal measure $\mu$ from the optimal dual variable $\eta$.
Evidently, atoms in $\mu$ correspond to points in the coincidence set and intervals in the support
of $\mu$ correspond to intervals in the coincidence set. 
Thus, one can bound the complexity of RSB by controlling the topology of the coincidence set. 
We explore this idea in the next section.

\subsection{A finite dimensional reduction of RSB}

We now turn to the main result of this paper, namely the reduction of
all possible RSB in spherical mixed $p$-spin glasses to a finite
parameter family. To this end, define 
\[
\mathfrak{d}=\left(\frac{1}{\sqrt{\xi''}}\right)''.
\]
Then we have the following theorem which controls the topology of the coincidence set.
\begin{thm}
\label{thm:Rule-of-Signs}(Rule of Signs) Let $\eta$ be optimal for \prettyref{eq:dual-prob}
and let $0\leq a<b<1$. We have the following cases:
\begin{enumerate}
\item If $\mathfrak{d}>0$ on $(a,b),$ then $card(\{\eta=\xi\}\cap[a,b])\leq2$.
\item If $\mathfrak{d}\leq0$ on $[a,b]$ and $a,b\in\{\eta=\xi\}$, then
$[a,b]\subset\{\eta=\xi\}$.
\end{enumerate}
\end{thm}
\begin{rem}
The first case can be seen on the primal side in the work of Talagrand \cite[Proposition 2.2]{TalSphPF06},
and also in \cite{CriSom92}. The second case is new and is the crux
of our reduction of RSB. 
\end{rem}

Combining this result with the duality theory from the previous section gives a bound on the complexity of RSB.
Observe that $\mathfrak{d}$ is either identically
zero or has finitely many roots in $[0,1]$ by the analyticity of
$\frac{3}{2}(\xi''')^{2}-\xi''\xi''''$ in the unit disc $B(0,1)\subset\C$. 
\begin{cor}
\label{cor:RSBreduction}Let $\{\mathfrak{d}>0\}=\cup_{i=1}^{N_{p}}(a_{i},b_{i})$
be the decomposition into connected components and similarly let $\{\mathfrak{d}\leq0\}=\cup_{i=1}^{N_{n}}[a_{i}',b_{i}']$.
Then, $\mu$ must be of the form
\[
\mu=\sum_{i=1}^{N_{p}}m_{i_{1}}\delta_{q_{i_{1}}}+m_{i_{2}}\delta_{q_{i_{2}}}+\sum_{i=1}^{N_{n}}-\mathfrak{d}\indicator{[r_{i_{1}},r_{i_{2}}]}ds+n_{i_{1}}\delta_{r_{i_{1}}}+n_{i_{2}}\delta_{r_{i_{2}}}
\]
where the points $q_{i_{1}},q_{i_{2}}\in[a_{i},b_{i}]$ and $r_{i_{1}},r_{i_{2}}\in[a_{i}',b_{i}']$.\end{cor}
\begin{rem}
See \cite[Section 1.2]{JagTobLT16} for several examples of this reduction
at zero-temperature (without the constraint that $\mu\in\Pr([0,1])$). \end{rem}
\begin{proof}
This follows by combining the optimality part of \prettyref{prop:dual-regularity}
with \prettyref{thm:Rule-of-Signs}. In particular, the first sum comes from the first part of the rule of signs. 
The second sum follows from the second part of the rule of signs,
combined with \prettyref{eq:optimaletaeqn} and, in particular, the
fact that if $(a,b)\subset \text{supp}\, \mu\subset\{\eta=\xi\}$ then $\mu$ has density $-\mathfrak{d}$ on that interval.
\end{proof}
As a quick application of \prettyref{cor:RSBreduction}, we demonstrate
how one can ascertain the nature of $\mu$ in the case that $\xi''(0)=0$.
In this case, one has that $\mathfrak{d}>0$ on a neighborhood of
$0$. (See the proof of \prettyref{thm:Rule-of-Signs} \vpageref{proof:rule-of-signs}.)
Thus, on that same neighborhood, there can be at most two atoms in
the support of $\mu$. %
A similar result was shown in \cite{AuffChen13}.

\subsection{Some remarks on Replica Symmetry}

We now comment briefly on the analysis of the RS region of the Crisanti-Sommers problem.
It is an important and natural question to determine the region in
$(\xi,h)$-space where the optimal measure is a single Dirac mass.
On the dual side, the RS region contains the region where the coincidence
set is a single point (given a positive solution of \prettyref{quest:shallowpts},
these sets would coincide). 

The standard method used to analyze RS is the ansatz-driven approach
which we recap now briefly. Assume that the coincidence set consists
of a single point $q$. Then by the optimality and consistency conditions \prettyref{eq:optimaletaeqn} and
\prettyref{eq:consistency}, this point must satisfy
\begin{align}
q& =  (\xi'(q)+h^{2})(1-q)^{2}\label{eq:fixedpointequation}\\
1& \geq  \xi''(q)(1-q)^{2}\label{eq:Repliconeigenvalue}
\end{align}
If one proposes $q$ which satisfies these relations, the question
reduces to that of determining whether or not the corresponding ansatz,
$\eta$, lies above the obstacle. For example at $h=0$, this is equivalent
to the condition of Talagrand \cite[Proposition 2.3]{TalSphPF06}. 

An important observation in these approaches is that the condition
of strict inequality in \prettyref{eq:Repliconeigenvalue} corresponds
to the positivity of what is called the ``replicon eigenvalue''
in the physics literature (cf.\ \cite[equation (4.9)]{CriSom92}).
It is an important and non-trivial question to determine for which
parameters this positivity is also sufficient. This is the question
of whether or not the de Almeida-Thouless line for the spherical model
is the boundary between the RS and RSB phases \cite{CriSom92}. The
analogous question for the Ising-spin models, where $\sigma\in\{-1,1\}^{N}$, has been the subject
of much research (see, e.g., \cite{TalBK11Vol2,Ton02,Chen15,JagTobPD15}). 

For an example of how the duality method can be used to study this problem,
consider the following. If $q$ solves the fixed point equation, then
for $t\leq q$, we have $\eta\geq\xi$ by positivity of the replicon
eigenvalue. It remains to show that this positivity implies the result
for $t\geq q$. Note that it suffices to prove that in fact 
\[
\eta''(t)\geq\xi''(t)\quad\forall t\geq q.
\]
Observe that this condition is implied by the inequality 
\[
1\geq\xi''(1)(1-q)^{2}.
\]
A straightforward manipulation of \prettyref{eq:fixedpointequation}
and \prettyref{eq:Repliconeigenvalue} shows that this holds provided
that $h^{2}\geq\xi''(1)$. This argument is the dual version of those
in \cite{JagTobPD15,TalBK11Vol2,Ton02}. 

One can also give ansatz-free arguments for RS. Consider the following
proof of RS for $\xi''(1)$ sufficiently small. (An identical result
is known for the Ising-spin models \cite{Chen15,JagTobPD15}.) Observe that by \prettyref{eq:consistency},
if there are two coincidence points, $a,b$, then they must satisfy
\[
\frac{1}{b-a}\int_{a}^{b}\eta''=\frac{\xi'(b)-\xi'(a)}{b-a}.
\]
The righthand side is upper bounded by $\xi''(1)$ by monotonicity
and the lefthand side is strictly lower bounded by $1$ by the regularity
part of \prettyref{prop:dual-regularity}. Thus for $\xi''(1)\leq1$
there must be exactly one coincidence point and one atom in the support
of $\mu$. A similar proof can be
given for the Ising-spin models using the fixed point equation and the
maximum principle for the Parisi PDE.

\subsection*{A discrepancy between scalings of spherical and Ising-spin models}

A natural question is to study the asymptotics of the Crisanti-Sommers
variational problem in various limits where $\beta\to\infty$ (see,
e.g., \cite{ChenSen15,AT78,JagTobLT16,JagTobPD15}). Taking the limit
$\beta\to\infty$ along the curve $\{\beta,h:\eta''(q)-\beta^{2}\xi_{0}''(q)=\Lambda\}$,
one finds, after manipulating \prettyref{eq:fixedpointequation} and
\prettyref{eq:Repliconeigenvalue}, that $\beta(1-q)\to c>0$. This
scaling, however, is different in the Ising-spin version of
this model, i.e., when we restrict $H_{N}$ to $\{-1,1\}^{N}$. In
particular, it was shown by the authors in \cite{JagTobLT16} that
along the equivalent positive replicon eigenvalue curve for the Ising
spin models, one has that $\beta^{2}(1-q)\to c>0$ for some $c$. Thus,
the large $\beta$ asymptotics in the two models appear to be fundamentally
different, at least in the RS regime.

\subsection{Some Open Questions}

Before turning to the body of this paper, let us first point out the
following natural questions which we leave for future research. A
major obstruction in the analysis of \prettyref{eq:Crisanti-Sommers-var-prob-def}
and similar problems (see, e.g., \cite{AuffChen13,JagTobPD15,TalBK11Vol2})
on the primal side is the existence of points at which the variational
derivative $G_{\mu}$ (in the notation of \cite{JagTobPD15}) achieves
its minimum value, but which are not in the support of $\mu$. In
the language of this paper, this is the question of whether or not
the set containment $\text{supp}\,\mu\subseteq\{\eta=\xi\}$ is strict.
Thus, an important question for future study in this field is:
\begin{quest}
\label{quest:shallowpts} For which $\xi$ and $h$ are the coincidence
set and the support of $\mu$ the same?
\end{quest}
As a first step toward the resolution of this question, we note the following proposition.
\begin{prop}
If $\mathfrak{d}$ is not identically zero, then the elements of $\{\eta=\xi\}\backslash \text{supp}\, \mu$
are isolated points. If $\mathfrak{d}$ is identically zero, then $\mu=m\delta_{r_{1}}+(1-m)\delta_{r_{2}}\ $
for $m\in[0,1]$ and $r_{1},r_{2}\in[0,1]$, and $\{\eta=\xi\}=[r_{1},r_{2}]$.
\end{prop}
\begin{proof}
Assume that $\mathfrak{d}$ is not identically zero. By the rule of
signs, the coincidence set is at most a union of isolated points and
non-trivial intervals. Thus it remains to show that any interval in
the coincidence set must be in the support of $\mu$. To this end,
observe that on any interval in the contact set, $\mu$ has density
$-\mathfrak{d}$ by \eqref{eq:optimaletaeqn}. By assumption, $\mathfrak{d}$
has only finitely many zeros so the interval must be in the support. 

Now assume that $\mathfrak{d}$ is identically zero. The result then
follows from \prettyref{cor:RSBreduction}.
\end{proof}
Another natural question is the following:
\begin{quest}
Find a probabilistic interpretation for the dual variable, $\eta$.
\end{quest}

\subsection{Acknowledgments}

We would like to thank  R.V.\ Kohn for helpful discussions. We would like to thank A.\ Auffinger for encouraging the preparation of this paper. We would like to thank D.\ Panchenko for helpful comments regarding the presentation of this work. This research was conducted while A.J.\ was supported by a National Science Foundation
Graduate Research Fellowship DGE-0813964; and National Science Foundation
grants DMS-1209165 and OISE-0730136, and while I.T.\ was supported
by a National Science Foundation Graduate Research Fellowship DGE-0813964;
and National Science Foundation grants OISE-0967140 and DMS-1311833.

\section{Proof of Duality}

In this section we prove \prettyref{thm:duality}. For readability,
we neglect to write the subscripts $\xi,h$ here and throughout the
remainder of the paper. We also abbreviate $\Pr=\Pr([0,1])$.

Define the set $Q=\{\mu\in\Pr:\sup\text{supp}\,\mu<1\}$ and recall
by \cite{TalSphPF06} that the minimizer of \prettyref{eq:Crisanti-Sommers-var-prob-def}
lies in $Q$. (For an alternative proof see \prettyref{lem:mu-in-Q}.)
Define the quantities 
\begin{align*}
\tilde{D}(\eta) & =\int_{0}^{1}2\sqrt{\eta''}-\eta''(s)\cdot(1-s)-\frac{1}{1-s}ds-\eta(0)+\inf_{t\in[0,1)}\left\{ \eta(t)-\xi(t)\right\} +h^{2}+\xi(1)\\
\tilde{P}(\mu) & =2P(\mu).
\end{align*}
 First, we observe the following integration by parts lemma, whose
proof we defer to the end of this section.
\begin{lem}
\label{lem:IBP}Let $\mu,\nu\in Q$ and $\eta\in X$. Then 
\[
\int\left(\eta''-\xi''\right)\left(\mbox{\ensuremath{\phi_{\nu}}-\ensuremath{\phi_{\mu}}}\right)\,dt=-\eta'(0)\cdot(\phi_{\nu}(0)-\phi_{\mu}(0))-\int\left(\eta-\xi\right)d(\nu-\mu).
\]

\end{lem}
We now begin the proof of \prettyref{thm:duality}. First, we prove
a preliminary duality result:
\begin{lem}
\label{lem:min-max}We have that 
\[
\min_{\mu\in Q}\tilde{P}(\mu)=\max_{\substack{\eta\in X\\
\eta'(0)=-h^{2}
}
}\tilde{D}(\eta).
\]
Furthermore an optimal pair $(\mu,\eta)$ satisfies 
\begin{align*}
\eta''\phi_{\mu}^{2} & =1\quad Leb-a.e.\\
\mu\left(\left\{ \eta(s)-\xi(s)=\inf_{t}\left(\eta(t)-\xi(t)\right)\right\} \right) & =1.
\end{align*}
Finally if a pair $(\mu,\eta)$ satisfies these relations then it
is optimal.\end{lem}
\begin{proof}
We begin by proving the first part. Recall that for all $a,\lambda\in(0,\infty)$
we have 
\begin{equation}
2\sqrt{\lambda}\leq\lambda a+\frac{1}{a}\label{eq:sqrt-legendre}
\end{equation}
with equality if and only if $\lambda a^{2}=1$. With this in mind,
we have for all $\mu\in Q$ and $\eta\in\Lambda$ that 
\begin{align*}
\tilde{P}(\mu) & \geq\int_{0}^{1}\xi''(s)\phi_{\mu}(s)+2\sqrt{\eta''}-\phi_{\mu}\eta''-\frac{1}{1-s}ds+h^{2}\phi_{\mu}(0)\\
 & =\int_{0}^{1}(\eta''(s)-\xi''(s))(1-s-\phi_{\mu}(s))+\int_{0}^{1}2\sqrt{\eta''}-\eta''(s)\cdot(1-s)-\frac{1}{1-s}\\
&\qquad +\int_{0}^{1}\xi''(s)\cdot(1-s)ds+h^{2}\phi_{\mu}(0).
\end{align*}

Applying \prettyref{lem:IBP} with $\nu=\delta_{0}$, we may integrate
by parts to find that the last line is equal to 
\[
K(\mu,\eta)=\int_{0}^{1}(\eta-\xi)d\mu+\int_{0}^{1}2\sqrt{\eta''}-\eta''(s)\cdot(1-s)-\frac{1}{1-s}ds-\eta(0)+h^{2}+\xi(1).
\]
That is, 
\[
\tilde{P}(\mu)\geq K(\mu,\eta).
\]
Observe that if we define $\eta_{\mu}$ by 
\begin{equation}
\eta_{\mu}=-h^{2}t+\int_{0}^{t}\int_{0}^{s}\frac{1}{\phi_{\mu}^{2}}d\tau ds,\label{eq:conj_eta}
\end{equation}
then it satisfies $\eta\in X$ and $\eta'(0)=-h^{2}$. By the case
of equality in \prettyref{eq:sqrt-legendre}, we see that (up to the
addition of a constant) it uniquely achieves the equality 
\[
\tilde{P}(\mu)=K(\mu,\eta_{\mu}).
\]
Hence, 
\[
\tilde{P}(\mu)=\sup_{\substack{\eta\in X\\
\eta'(0)=-h^{2}
}
}\,K(\mu,\eta).
\]

Evidently, 
\[
\inf_{\mu\in Q}\,K(\mu,\eta)=\tilde{D}(\eta).
\]
This implies that 
\[
\inf_{\mu\in Q}\,\tilde{P}(\mu)=\inf_{\mu\in Q}\sup_{\substack{\eta\in X\\
\eta'(0)=-h^{2}
}
}\,K(\mu,\eta)\geq\sup_{\substack{\eta\in X\\
\eta'(0)=-h^{2}
}
}\inf_{\mu\in Q}\,K(\mu,\eta)=\sup_{\substack{\eta\in X\\
\eta'(0)=-h^{2}
}
}\,\tilde{D}(\eta).
\]
The first part of the result then follows provided we can find a pair
$(\mu,\eta)$ that achieves the equality 
\[
\tilde{P}(\mu)=K(\mu,\eta)=\tilde{D}(\eta).
\]
 Let $\mu$ be the optimizer of $P$ and define $\eta=\eta_{\mu}$
as in \prettyref{eq:conj_eta}. The result then follows provided 
\[
K(\mu,\eta)=\tilde{D}(\eta).
\]
Observe that this will happen if and only if
\begin{equation}
\int\eta-\xi\,d\mu=\inf\,\{\eta-\xi\}\label{eq:eq_d}
\end{equation}
so it suffices to show this equality. To see \prettyref{eq:eq_d},
note that the first order optimality conditions for \prettyref{eq:Crisanti-Sommers-var-prob-def}
are that $\mu$ satisfies 
\[
\int\left(\xi''-\frac{1}{\phi_{\mu}^{2}}\right)(\phi_{\nu}-\phi_{\mu})dt+h^{2}(\phi_{\nu}(0)-\phi_{\mu}(0))\geq0\quad\forall\,\nu\in Q.
\]
Identifying the integrand, we see that this is 
\begin{equation}
h^{2}(\phi_{\nu}(0)-\phi_{\mu}(0))-\int\left(\eta''-\xi''\right)(\phi_{\nu}-\phi_{\mu})dt\geq0.\label{eq:eq_b}
\end{equation}
Integrating by parts (see \prettyref{lem:IBP}), we see that 
\[
\int\left(\eta''-\xi''\right)\left(\mbox{\ensuremath{\phi_{\nu}}-\ensuremath{\phi_{\mu}}}\right)dt=h^{2}(\phi_{\nu}(0)-\phi_{\mu}(0))-\int\left(\eta-\xi\right)d(\nu-\mu).
\]
Combining this with \prettyref{eq:eq_b} yields 
\[
\int\eta-\xi d\nu\geq\int\eta-\xi d\mu
\]
for all $\nu\in Q$. The first result is then immediate. 

Observe by the construction of $\eta_{\mu}$ and the preceding discussion
that the optimal pair constructed satisfies the above relations. Since
the two conditions \prettyref{eq:eq_d} and that $\eta''\phi^{2}=1$
were necessary and sufficient for the equalities to hold, we see that
the optimality conditions characterize optimal pairs as desired.
\end{proof}
We can now give the proof of the duality. 

\begin{proof}[Proof of \prettyref{thm:duality}] \label{proof:proofofduality}
Observe that by \prettyref{lem:min-max}, and the fact that the minimizer
of \prettyref{eq:Crisanti-Sommers-var-prob-def} belongs to $Q$ (see
\prettyref{lem:mu-in-Q}), we have that 
\begin{align*}
\min_{\mu\in\Pr}P(\mu) & =\max_{\substack{\eta\in X\\
\eta'(0)=-h^{2}
}
}\frac{1}{2}\tilde{D}(\eta).
\end{align*}
So, it suffices to show that 
\[
\max_{\substack{\eta\in X\\
\eta'(0)=-h^{2}
}
}\frac{1}{2}\tilde{D}(\eta)=\max_{\eta\in\Lambda}D(\eta).
\]
To prove this, we begin by observing that $\tilde{D}$ is invariant
under shifts $\eta\mapsto\eta+k$. Choose an optimal $\eta$ for $\frac{1}{2}\tilde{D}$
(we know an optimizer exists from the proof of \prettyref{lem:min-max}
above). Then if we set $\tilde{\eta}=\eta-\inf\left\{ \eta-\xi\right\} $,
it follows that 
\[
\frac{1}{2}\tilde{D}(\eta)=\frac{1}{2}\tilde{D}(\tilde{\eta})=D(\tilde{\eta}).
\]
Indeed,  $\tilde{\eta}\geq\xi$ and $\inf\,\{\tilde{\eta}-\xi\}=0$. 

That the pair is unique can be seen as follows. Observe first that
$P$ is strictly convex so that the minimizing $\mu$ is unique. For
$D$, note that although it is concave, it is not strictly concave.
Nevertheless, if $\eta_{1}$ and $\eta_{2}$ were maximizers, then
they would have the same second derivative by strict convexity of
the square root function. As their derivative at $0$ would be prescribed
in the definition of $\Lambda$, we see that they could differ only
by a constant. If this constant is non-zero, then this would change
the value of the term $\eta(0)$ appearing in the definition of $D$,
which would be a contradiction. Thus, the maximizing $\eta$ is unique.
That the pair is uniquely determined by the relations \prettyref{eq:opt-conds}
follows from \prettyref{lem:min-max}.

\end{proof}

Finally, we prove the integration by parts lemma that was used above.

\begin{proof}[Proof of \prettyref{lem:IBP}]

By an approximation argument, it suffices to prove \prettyref{lem:IBP}
assuming that $\mu=fds$ and $\nu=gds$ where $f,g\in C_{c}^{\infty}((0,1))$,
and that $\eta\in C_{c}^{\infty}([0,1))$. Given these, it follows
that $\phi_{\mu}-\phi_{\nu}\in C_{c}^{\infty}([0,1))$. Hence, integrating
by parts gives that 
\begin{align*}
\int_{0}^{1}(\eta''-\xi'')(\phi_{\nu}-\phi_{\mu})\,ds & =\int_{0}^{1}-(\eta-\xi)'(\phi_{\nu}'-\phi_{\mu}')\,ds-\eta'(0)(\phi_{\nu}(0)-\phi_{\mu}(0))\\
 & =\int_{0}^{1}(\eta-\xi)(\phi_{\nu}''-\phi_{\mu}'')\,ds-\eta'(0)(\phi_{\nu}(0)-\phi_{\mu}(0))+\eta(0)(\phi_{\nu}'(0)-\phi_{\mu}'(0)).
\end{align*}
Note we used that $\xi(0)=\xi'(0)=0$. Given our assumptions on $\mu$
and $\nu$, 
\[
\int_{0}^{1}(\eta-\xi)(\phi_{\nu}''-\phi_{\mu}'')\,ds=-\int_{0}^{1}(\eta-\xi)\,d(\nu-\mu)
\]
and $\phi_{\nu}'(0)=\phi_{\mu}'(0)=0$. Plugging these into the above
gives the result.

\end{proof}

\section{Proof of the rule of signs}

Throughout this section, we assume that $\eta\in\Lambda$ is optimal,
i.e., that
\[
D(\eta)=\max_{\eta\in\Lambda}D(\eta).
\]
Our goal is to prove \prettyref{thm:Rule-of-Signs}. 

\begin{proof}[Proof of \prettyref{thm:Rule-of-Signs}] \label{proof:rule-of-signs}

We begin by proving (1), which we do by contradiction. Suppose that
there are three coincidence points, $t_{1}$, $t_{2}$, and $t_{3}$,
which satisfy $a\leq t_{1}<t_{2}<t_{3}\leq b$. By \prettyref{eq:consistency},
$\eta'\left(t_{i}\right)=\xi'\left(t_{i}\right)$ for each $i$, and
$\eta''\left(t_{3}\right)\geq\xi''\left(t_{3}\right)$. By the mean
value theorem, there exist points $s_{1},s_{2}$ with $t_{1}<s_{1}<t_{3}<s_{2}<t_{2}$
such that $\eta''\left(s_{i}\right)=\xi''\left(s_{i}\right)$ for
$i=1,2$. By \prettyref{eq:optimaletaeqn}, the function $t\to\frac{1}{\sqrt{\eta''\left(t\right)}}$
is concave. Hence, by the assumption on $\mathfrak{d}$, $t\to\frac{1}{\sqrt{\eta''\left(t\right)}}-\frac{1}{\sqrt{\xi''\left(t\right)}}$
is strictly concave on $(t_{1},t_{2})$. As this function vanishes
at the points $s_{1},s_{2}$, it must be strictly positive between.
Thus, $\xi''\left(t_{2}\right)>\eta''\left(t_{2}\right)$ and this
is a contradiction.

Now we prove (2). We proceed by a first variation argument. Introduce
the path 
\[
[0,1]\to\Lambda,\quad\tau\to\eta_{\tau}=\eta+\tau(\tilde{\eta}-\eta)
\]
where 
\begin{equation}
\tilde{\eta}\left(t\right)=\begin{cases}
\eta\left(t\right) & t\notin\left[a,b\right]\\
\xi\left(t\right) & t\in\left[a,b\right]
\end{cases}.\label{eq:eta-tilde}
\end{equation}
That $\eta_{\tau}\in\Lambda$ for all $\tau$, and in particular that
$\eta_{\tau}\in X$, follows from our assumption that $\eta$ is optimal
and that $a,b\in\{\eta=\xi\}$. In particular, we have that $\eta'=\xi'$
at $a,b$ by \prettyref{eq:consistency}. 

Now we note that $D$ is concave and that the path $\eta_{\tau}$
is linear in $\tau$, so that $\tau\to D(\eta_{\tau})$ is concave.
Thus, to conclude that $[a,b]\subset\{\eta=\xi\}$, we need only to
prove that 
\begin{equation}
\at{\frac{d}{d\tau}^{-}}_{\tau=1}D(\eta_{\tau})\geq0,\label{eq:non-negderivative}
\end{equation}
since it then follows that 
\[
D(\tilde{\eta})=D(\eta_{1})\geq D(\eta_{0})=D(\eta)
\]
so that $\tilde{\eta}$ is also a maximizer of $D$. Since maximizers
of $D$ are unique (as shown in the proof of \prettyref{thm:duality}
\vpageref{proof:proofofduality}), we conclude that $\tilde{\eta}=\eta$
and hence that $[a,b]\subset\{\eta=\xi\}$.

Now we perform the required differentiation. First, note that
\[
\eta_{\tau}(0)=\eta(0)\quad\forall\,\tau.
\]
Also, note that 
\[
\int_{a}^{b}(\eta_{\tau}''-\eta'')(1-s)\,ds=(\eta_{\tau}'-\eta')(1-s)|_{a}^{b}+(\eta_{\tau}-\eta)|_{a}^{b}=0\quad\forall\,\tau
\]
by \prettyref{eq:consistency}. It follows that 
\[
D(\eta_{\tau})-D(\eta)=\int_{a}^{b}\sqrt{\eta_{\tau}''}-\sqrt{\eta''}\,ds\quad\forall\,\tau
\]
so that
\[
D(\eta_{1})-D(\eta_{\tau})=\int_{a}^{b}\sqrt{\eta_{1}''}-\sqrt{\eta_{\tau}''}=\int_{a}^{b}\sqrt{\xi''}-\sqrt{\xi''+(1-\tau)(\eta''-\xi'')}.
\]
Hence,
\[
\at{\frac{d}{d\tau}^{-}}_{\tau=1}D(\eta_{\tau})=\lim_{\tau\to1^{-}}\frac{D(\eta_{1})-D(\eta_{\tau})}{1-\tau}=\int_{a}^{b}-\frac{1}{2}\frac{1}{\sqrt{\xi''}}(\eta''-\xi'').
\]
Here we used the monotone convergence theorem to pass to the limit,
which we may do since $\eta''-\xi''$ changes sign only finitely many
times.

Now we claim that 
\begin{equation}
\at{\frac{d}{d\tau}^{-}}_{\tau=1}D(\eta_{\tau})=\frac{1}{2}\int_{a}^{b}-\mathfrak{d}(\eta-\xi).\label{eq:IBPresult}
\end{equation}
We prove this by case analysis. Suppose first that $a>0$. This then
follows by an integration by parts along with \prettyref{eq:consistency}. 

Now suppose that $a=0$. We claim that when $\mathfrak{d}\leq0$ on
this interval, that \prettyref{eq:IBPresult} still holds by an integration
by parts. Indeed, we claim that under this assumption that $\xi''(0)>0$.
With this claim in hand, integration by parts and \prettyref{eq:consistency}
gives that
\[
\int_{\epsilon}^{b}\frac{1}{\sqrt{\xi''}}(\eta''-\xi'')=\int_{\epsilon}^{b}-\mathfrak{d}(\eta-\xi)-\frac{1}{\sqrt{\xi''}}(\eta'-\xi')(\epsilon)+(\frac{1}{\sqrt{\xi''}})'(\eta-\xi)(\epsilon)
\]
for all $\epsilon\in(0,b)$. Since $\xi''(0)>0$, $\frac{1}{\sqrt{\xi''}}\in L^{1}([0,b])$
and also $\frac{1}{\sqrt{\xi''}}\vee|(\frac{1}{\sqrt{\xi''}})'|\leq C$.
Since $\eta\in C^{2}([0,1))$ by \prettyref{prop:dual-regularity},
we can take $\epsilon\to0$ to deduce \prettyref{eq:IBPresult}.

We now prove the claim. Suppose, for contradiction, that $\xi''(0)=0$.
By direct manipulation, $\mathfrak{d}$ has the same sign as 
\[
s(t)=3(\xi''')^{2}-2\xi''\cdot\xi''''
\]
for $t\in(0,1)$. Let $k=\min\{p:\beta_{p}>0\}$ then 
\[
s(t)=\beta_{k}^{4}k^{3}(k-1)^{2}(k-2)t^{2k-6}+O(t^{2k-5}).
\]
This is positive in a neighborhood of zero, contradicting the assumption
that $\mathfrak{d}\leq0$.

To complete the proof, observe that the integrand in \prettyref{eq:non-negderivative}
is non-negative by our assumption on $\mathfrak{d}$ and since $\eta\geq\xi$.
Hence, \prettyref{eq:non-negderivative} follows and therefore $[a,b]\subset\{\eta=\xi\}$. 

\end{proof}

We conclude with the following remark.
\begin{rem}
\label{rem:We-are-awesome}This proof is an example of the observation
alluded to in the introduction that there are methods available to
the study of the dual problem that do not have apparent analogues
for the primal one. In particular, the variation used above defined through
\prettyref{eq:eta-tilde} does not obviously have an associated path
on the primal side, as it only becomes apparent that $\left(\frac{1}{\sqrt{\tilde{\eta}''}}\right)''$
is a probability measure \emph{a posteriori}!
\end{rem}

\section{Appendix}

To make this presentation self-contained we present an alternative
proof of the following statement.
\begin{lem}
\label{lem:mu-in-Q} The minimizer of \prettyref{eq:Crisanti-Sommers-var-prob-def}
belongs to $Q=\{\mu\in\Pr:\sup\text{supp}\,\mu<1\}$.\end{lem}
\begin{proof}
It suffices to prove the following claim: given $\mu\in\Pr\backslash Q$,
there exists $\tilde{\mu}\in Q$ with $P(\tilde{\mu})<P(\mu)$. We
prove this by a first variation argument. Given $\mu\in\Pr$, consider
the variation $[0,1]\to\Pr$, $\epsilon\to\mu_{\epsilon}$ defined
by
\[
\mu_{\epsilon}([0,t])=\begin{cases}
\mu([0,t]) & 0\leq t<1-\epsilon\\
1 & 1-\epsilon\leq t\leq1
\end{cases}.
\]
Note that $\mu_{\epsilon}\in Q$ if $\epsilon>0$. Evidently
\[
\phi_{\epsilon}=\phi_{\mu_{\epsilon}}=\begin{cases}
\phi_{\mu}(t)-\phi_{\mu}(1-\epsilon)+\epsilon & 0\leq t<1-\epsilon\\
1-t & 1-\epsilon\leq t\leq1
\end{cases},
\]
so that 
\[
\phi_{\epsilon}-\phi=\begin{cases}
\epsilon-\phi_{\mu}(1-\epsilon) & 0\leq t<1-\epsilon\\
1-t-\phi_{\mu}(t) & 1-\epsilon\leq t\leq1
\end{cases}.
\]
Hence,
\begin{align*}
P(\mu_{\epsilon})-P(\mu) & =\int_{0}^{1}(\xi''-\frac{1}{\phi_{\epsilon}\phi})(\phi_{\epsilon}-\phi)\,ds+h^{2}(\phi_{\epsilon}(0)-\phi(0))\\
 & =(\epsilon-\phi_{\mu}(1-\epsilon))\left[\xi'(1-\epsilon)+h^{2}-\int_{0}^{1-\epsilon}\frac{1}{\phi_{\epsilon}\phi}\right]+\int_{1-\epsilon}^{1}(\xi''-\frac{1}{\phi_{\epsilon}\phi})(1-t-\phi)\,ds\\
 & =(i)+(ii).
\end{align*}

Now suppose that $\mu\in\Pr\backslash Q$, so that $\phi_{\mu}(1-\epsilon)<\epsilon$
for all $\epsilon>0$. Using the upper bound $\phi_{\epsilon}\vee\phi\leq1-t$
and that $\xi'$ is non-decreasing, we see that
\[
\xi'(1-\epsilon)+h^{2}-\int_{0}^{1-\epsilon}\frac{1}{\phi_{\epsilon}\phi}\leq\xi'(1)+h^{2}-\int_{0}^{1-\epsilon}\frac{1}{(1-t)^{2}}=\xi'(1)+h^{2}-\frac{1-\epsilon}{\epsilon}.
\]
The righthand side is strictly negative for small enough $\epsilon$,
so $(i)$ is strictly negative for small enough $\epsilon$. Similarly,
we see that $(ii)$ is strictly negative for sufficiently small $\epsilon>0$.
It follows that $P(\mu_{\epsilon})-P(\mu)<0$ for small enough $\epsilon$.
\end{proof}
\bibliographystyle{plain}
\bibliography{finitebeta}

\end{document}